\documentclass[a4paper,12pt]{amsart}

\usepackage{amssymb,latexsym,amsmath,amsthm,amsfonts, mathrsfs, enumerate}
\usepackage{xcolor}
\usepackage[colorlinks]{hyperref}

\newtheorem{theorem}{Theorem}[section]
\newtheorem{lemma}[theorem]{Lemma}

\newtheorem{proposition}[theorem]{Proposition}

\newtheorem{definition}[theorem]{Definition}

\newtheorem{remark}[theorem]{Remark}
\theoremstyle{remark}
\theoremstyle{definition}
\newcommand{\cb}{\mathcal B}

\newcommand{\br}{\mathbb{R}}

\newcommand{\gtg}{G \times \widehat{G}}

\newcommand{\wo}{\widehat{\otimes}}
\allowdisplaybreaks

\title[Characterization of Weyl multipliers]{Characterization of vector-valued $L^1$-$L^p$ multipliers for  Weyl transform}
\author{Ritika Singhal}
\address{Ritika Singhal \newline \hspace*{0.25cm} 
	Department of Mathematics\newline \hspace*{0.25cm}
	Indian Institute of Technology Delhi\newline \hspace*{0.25cm}
	Delhi - 110016\newline \hspace*{0.25cm}
	India.}
\email{ritikasinghal1120@gmail.com}
\author{N. Shravan Kumar}
\address{N. Shravan Kumar \newline \hspace*{0.25cm}
	Department of Mathematics\newline \hspace*{0.25cm}
	Indian Institute of Technology Delhi\newline \hspace*{0.25cm}
	Delhi - 110016\newline \hspace*{0.25cm}
	India.}
\email{shravankumar.nageswaran@gmail.com}



\begin{document}
	\maketitle
	
	\begin{abstract}
		In this article, we will characterize Weyl multipliers for the pair $(L^1(G \times \widehat{G}; A), L^p(G \times \widehat{G};A))$,  for $1 \leq p< \infty$, under the assumption that $A$ is a complex Banach algebra with a bounded approximate identity.\end{abstract}	
		
		\section{Introduction}
		One of the classical problems in harmonic analysis is the study of the Fourier multipliers, which aims to characterize the behaviour of certain bounded operators on function spaces using the Fourier transform. The well-known H\"ormander multiplier theorem \cite{Hor} provided sufficient conditions for a measurable function to be a $L^p$-$L^q(\br^n)$ Fourier multiplier for $1 \leq p,q < \infty$. One of the well-known results in characterizing $L^1$-$L^p$ multipliers for a locally compact abelian group $G$,
		states that an operator $T$ from $L^1(G)$ to $L^p(G), \  1 \leq p < \infty$, which commutes with translations is given by $Tf=f * \nu$ for $\nu \in M(G)$ if $p=1$ and  $\nu \in L^p(G)$ if $1<p <\infty$. More generally, we have the following result for $L^1$-multipliers.
		\begin{theorem}[\cite{MR435738}, Theorem 0.1.1] \label{FMT}
		Let $G$ be a locally compact abelian group, and suppose $T: L^1(G) \rightarrow L^1(G)$ is a continuous linear transformation. Then the following are equivalent: \begin{enumerate}
			\item[(i)] $T$ commutes with the translation operators.
			\item[(ii)] $T(f * g)=T f * g$ for each $f, g \in L^1(G)$.
			\item[(iii)] There exists a unique function $\phi$ defined on $\widehat{G}$ such that $(T f)^{\wedge}=\phi \hat{f}$ for each $f \in L^1(G)$.
			\item[(iv)] There exists a unique measure $\mu \in M(G)$ such that $(T f)^{\wedge}=\hat{\mu} \hat{f}$ for each $f \in L^1(G)$.
			\item[(v)] There exists a unique measure $\mu \in M(G)$ such that $T f=f * \mu$ for each $f \in L^1(G)$.
		\end{enumerate}
		\end{theorem}
		Over the last few decades, exploring vector-valued functions has gained widespread popularity.  Akinyele \cite{MR435734} considered the problem of providing a vector version of the characterizations of
		the multipliers for $L^1(G)$ under the assumption that Banach algebra is commutative and semi-simple. Chan further considered this problem in \cite{chan} 
		for vector-valued $L^1$-multipliers on a locally compact group for a Banach algebra with a bounded approximate identity. Later, Radha in \cite{MR1377370} studied $L^1$-$L^p$ multipliers for commutative complex Banach algebras with a bounded
		approximate identity.

		In \cite{MR545325}, Michele and Mauceri proved a version of the H\"ormander’s multiplier theorem for the Heisenberg group. The study of the 
		Heisenberg group is closely related to the study of the Weyl transform introduced by  Hermann Weyl in \cite{Weyl}. It is a pseudo-differential operator associated to a measurable function on $\mathbb{R}^n \times \mathbb{R}^n$, which has applications both in quantum mechanics and partial differential equations. Mauceri \cite{Mau} introduced the concept of the Weyl multipliers and proved an analogue of the
		H\"ormander’s multiplier theorem. Thangavelu in \cite{Tha} studied the Weyl multipliers using Littlewood-Paley-Stein theory and also obtained a characterization for $L^p$-Weyl multipliers using pseudo measures in \cite{RT}. In the last few decades, there have been several other studies for the multiplier theory of the Heisenberg group and Weyl multipliers \cite{MR4280273,MR2484216, MR2319602, RsNK}. The problem gets interesting because, unlike the Fourier transform,
		the multiplier associated with the Weyl transform is
		vector-valued.

		Using twisted convolution and Radon measures, a characterization similar to Theorem \ref{FMT} was obtained in \cite{RSK} for the Weyl multipliers.
		Motivated by this, we will provide a vector-valued characterization of the Weyl multipliers, assuming that the Banach algebra under consideration has a bounded approximate identity. It is important to mention that, unlike the classical case, the twisted convolution defined for the Weyl transform is not commutative, even though the underlying group and the Banach algebra are both commutative.

		\section{Preliminaries}
		Let $G$ be a locally compact abelian group with $\widehat{G}$ as its dual group. For $1\leq p < \infty$, we will denote by $L^p(G)$, the classical $L^p$-spaces on $G$ with respect to the Haar measure on $G$.  $C_c(G)$ represents the space of all compactly supported functions on $G$ and $\mathcal{B}(L^2(G))$ denotes the Banach algebra of all bounded operators on $L^2(G)$.
		
		The \textit{Weyl transform}, denoted  $W,$ is defined as a $\mathcal{B}\left(L^{2}(G)\right)$-valued integral on $C_{c}(G \times \widehat{G})$ given by
		$$W(f) \varphi(y)=\int_{G \times \widehat{G}}f(x, \chi) \rho((x, \chi))(\varphi)(y) dx  d\chi, f \in C_c(G \times \widehat{G})$$ 
		where $\rho(x,\chi)$ is motivated from the \textit{Schr\"odinger representation} of the abstract Heisenberg group $\mathbb{H}(G)$ on $L^2(G)$ and defined as $$\rho((x, \chi))(\varphi)(y)= \chi(y) \varphi(x y), \varphi \in L^{2}(G).$$
		By density of compactly supported functions, the Weyl transform can be defined for $L^p(\gtg), 1 \leq p < \infty$. Similarly, the Weyl transform for any bounded regular
		complex Borel measure $\nu$ in $M(\gtg)$ can be defined as a bounded map on $L^2(G)$ given by
		$$W(\nu) \phi(y)=\int_{\gtg} (\rho(x,\chi) \phi)(y) \ d \nu(x,\chi), \phi \in L^2(G).$$
		For more details on Weyl transform, one can refer to \cite{F, RSK,Wong}.

		Let $\mathcal{H}$ be a  Hilbert space and $1\leq p<\infty.$ If  $T:\mathcal{H}\rightarrow \mathcal{H}$ is a compact operator, then it admits an orthonormal representation $$T=\underset{n\in\mathbb{N}}{\sum}S_T(n)\left \langle .,e_n\right \rangle  \sigma_n,$$ where $\{e_n\}$ and $\{\sigma_n\}$ are orthonormal sequences in $\mathcal{H}$ and $S_T(n)$ denotes the $n^{th}$-singular value of $T$. The \textit{$p^{th}$-Schatten-von Neumann class}, denoted $\mathcal{B}_p(\mathcal{H}),$ consists of all compact operators, $T:\mathcal{H}\rightarrow\mathcal{H}$ with $\{S_T(n)\}\in\ell^p.$
		For $T\in\mathcal{B}_p(\mathcal{H}),$ define a norm$$\|T\|_{\mathcal{B}_p(\mathcal{H})}:=\left(\underset{n\in\mathbb{N}}{\sum}|S_T(n)|^p\right)^{1/p}= \|\{S_T(n)\}\|_{l^p}.$$ The space $\mathcal{B}_p(\mathcal{H})$ with the above norm becomes a Banach space. We will denote by $\mathcal{B}_\infty(\mathcal{H})$, the space of all compact operators on $\mathcal{H}.$  
		
		In order to define an algebra structure on $C_c(G \times \widehat{G})$, we will now define the concept of twisted convolution. 
		For $f, g \in C_{c}(G \times \widehat{G}),$ the {\it twisted convolution}, denoted $f \times g,$ is defined as
		\begin{align*}
		f \times g(x, \chi):= &\int_{G \times \widehat{G}} f\left(x x'^{-1}, \chi \overline{\chi'}\right) g\left(x^{\prime}, \chi^{\prime}\right)\chi'(x x'^{-1}) d x' d  \chi' \\= &\int_{G \times \widehat{G}} f\left(x^{\prime}, \chi^{\prime}\right) g\left(x x'^{-1}, \chi \overline{\chi'}\right) \left( \chi \overline{\chi'} \right)(x') d x' d  \chi' .		
		\end{align*}
		The following theorem gives the integrability properties of the Weyl transform with respect to the twisted convolution structure. For $1 \leq p< \infty,$  throughout $ p'$  will denote the conjugate exponent of $p$.
		\begin{theorem} \cite{RSK}
		\begin{enumerate}\item[(i)] For $p\in\{1,2\}$, the space $L^p(G \times \widehat{ G})$ is a Banach algebra with pointwise addition and twisted convolution as addition and multiplication respectively.
			\item[(ii)] (Plancherel Theorem) The Weyl transform is an isometric Banach algebra isomorphism between $L^2(G\times\widehat{G})$ and $\mathcal{B}_2(L^2(G)).$
			\item[(iii)]  (Riemann-Lebesque Lemma) The Weyl transform is a bounded operator and a Banach algebra homomorphism from $L^1(G \times \widehat{G})$ to $\mathcal{B}_\infty\left(L^{2}(G)\right).$
			\item[(iv)] (Hausdorff-Young inequality) If $1<p<2,$ then the Weyl transform maps $L^p(G\times\widehat{G})$ into $\mathcal{B}_{p'}(L^2(G))$.
		\end{enumerate}
		\end{theorem}
		To simplify the twisted convolution, we now introduce the notion of twisted translation, an analogue for the classical translation. For $(x, \chi) \in G \times \widehat{G}$, the twisted translates of $f: G \times \widehat{G} \rightarrow \mathbb{C}$ are defined as follows
		\begin{align*}
		\left( T^t_{(x',\chi')}f\right)(x,\chi )=& f\left(x x'^{-1}, \chi \overline{\chi'}\right) \chi'(x x'^{-1}); \\
		\left( L^t_{(x',\chi')}f\right)(x,\chi )=&f\left(x x'^{-1}, \chi \overline{\chi'}\right) \chi \overline{\chi'}(x') ,
		\end{align*}
		so that
		\begin{align*}
		f \times g(x, \chi):= &\int_{G \times \widehat{G}} 	\left(T^t_{(x',\chi')}f \right)(x,\chi ) g\left(x^{\prime},\chi^{\prime}\right)  d x' d  \chi' \\;
  = &\int_{G \times \widehat{G}} f\left(x^{\prime}, \chi^{\prime}\right) 	\left( L^t_{(x',\chi')}g \right)(x,\chi )d x' d  \chi'.	
		\end{align*} 
		Clearly, for $1\leq p < \infty$, twisted translation is an isometry from $L^p(\gtg)$ to itself. 
		Note that the need to define the operator $L^t$ arises because, unlike the classical translation, 
		$$\int_{G \times \widehat{G}} 	\left(T^t_{(x',\chi')}f \right)(x,\chi ) g\left(x^{\prime},\chi^{\prime}\right)  d x' d  \chi' \neq \int_{G \times \widehat{G}} f\left(x^{\prime}, \chi^{\prime}\right) 	\left( T^t_{(x',\chi')}g \right)(x,\chi )d x' d  \chi'.$$
		The following lemma shows that the twisted translation operator occurs naturally in the study of the Weyl transform.
		
		\begin{lemma}
		For $(x, \chi) \in G \times \widehat{G}$ and $f \in L^p(G \times \widehat{G}), 1 \leq p < \infty $, we have 
		
		\begin{enumerate}
			\item[(i)] 	$W\left(T_{(x, \chi)}^t f\right)=W(f) \rho(x, \chi)$;\\
			\item[(ii)]$ W\left(L_{(x, \chi)}^t f\right)=\rho(x, \chi)W(f) $.
		\end{enumerate}
		
		\end{lemma}
		\begin{proof} 
		For any $\varphi \in L^2(G)$, we have
		\begin{align*} (W(f)\rho(x,\chi))(\varphi)(y)=& 
			\int_{G \times \widehat{G}}f(x',\chi') (\rho(x',\chi'))(\rho(x,\chi)\varphi)(y)dx' d\chi' \\
			=& \int_{G \times \widehat{G}}f(x',\chi') \chi'(y)\left(\rho(x,\chi)\varphi\right)(x'y)dx' d\chi'\\
			=& \int_{G \times \widehat{G}}f(x',\chi') \chi'(y)\chi(x'y)\varphi(xx'y)dx' d\chi' .\end{align*}
		Consider the map $(x',\chi') \to (x'x^{-1},\chi' \overline{\chi}).$ Since  the measure on  $G \times \widehat{G}$ is Haar measure, we have 
		\begin{align*}
			(W(f)\rho(x,\chi))(\varphi)(y)=&  \int_{G \times \widehat{G}} f(x'x^{-1},\chi' \overline{\chi}) \left( \chi' \overline{\chi}\right)(y) \chi(x'x^{-1}y)\varphi(x'y)dx' d\chi'\\
			=& \int_{G \times \widehat{G}} f(x'x^{-1},\chi' \overline{\chi}) \chi'(y) \chi(x')\overline{\chi}(x)\varphi(x'y)dx' d\chi' \\
			=& \int_{G \times \widehat{G}}\left(f(x'x^{-1},\chi' \overline{\chi}) \chi(x')  \overline{\chi}(x) \right ) \chi'(y)\varphi(x'y) dx' d\chi' \\=& \int_{G \times \widehat{G}} \left( T^t_{(x,\chi)}f \right)(x',\chi')(\rho(x',\chi'))(\varphi)(y)dx' d\chi'  \\=& (W(T^t_{(x,\chi)}f))(\varphi)(y).\end{align*}
		Similarly, one can prove $(ii)$.
		\end{proof}
		The existence of a bounded approximate identity in $L^1(\gtg)$ can be proved as in \cite[Proposition 2.44]{FolAHa}.

		Let $A$ be a Banach space and let $A^*$ denote its Banach space dual. For $1\leq p < \infty$, let $L^p(G;A)$ denotes the space of all $p$-Bochner integrable function from $G$ to $A$, see  \cite{HTJM, MR2610316}.   
	For any complex Banach algebra $A$, the Weyl transform of a Bochner integrable function can be seen as a $\cb(L^2(G; A))$-valued bounded map on $L^p(\gtg; A), \  1 \leq p \leq 2.$  Similarly, the Weyl transform for any vector measure $\nu \in M(\gtg;A)$ can also be defined. Since $A$ is a Banach algebra, the concepts of twisted convolution and twisted translation can be defined naturally for Bochner integrable functions on $\gtg$.

  	The following lemma will be used to study the product between functionals in $A^*$ and functionals in $A^{**}$ and define an algebra structure in $A^{**}$. We will denote the canonical image of $a \in A$ inside $A^{**}$ by $a$ itself.
	\begin{lemma}[\cite{aren}]
		Let $A$ be a Banach algebra. Let $a,b \in A, f \in$ $A^*, F, G \in A^{**}$.
		Then $A^{**}$ is a Banach algebra under the Arens product,  denoted by $\odot$, defined as follows.
		\begin{enumerate}[(i)]
			\item 	Define $ a \odot f \in  A^*$ as $	(a \odot f)(b)  =f(ba). $ 
			\item	Define $f \odot F \in A^*$ as $(f \odot F)(a)  =F(a \odot f)$.
			\item Define $F \odot G \in A^{* *}$ as $(F \odot G)(f)  =G(f \odot F).$
		\end{enumerate}
	\end{lemma}

		For $1 \leq p,q < \infty$, Mauceri in \cite{Mau} defined $L^p$-$L^q$ Weyl multiplier as a bounded operator $M \in \mathcal{B}(L^2(G))$ such that the operator $C_M $ defined  by $W(C_Mf)=MW(f)$  extends to bounded operator from $L^p(G \times \widehat{G})$ to $L^q(G \times \widehat{G})$. The following characterization of $L^1$-$L^p$ multipliers, for  $1 \leq p< \infty$, was obtained in \cite{RsNK}.
		\begin{theorem} \label{Weyl m}
			Let $T$ be a bounded operator from $L^1(G \times \widehat{G})$ to $L^p(G \times \widehat{G})$, $1 \leq p < \infty$ . Then the following are equivalent:		\begin{enumerate}			\item[(i)] $T$ commutes with twisted translations.
				\item[(ii)] $T(f \times g)=T(f )\times g$ for all $f, g \in L^1(G \times \widehat{G})$.
				\item[(iii)] there exists a unique measure $\mu \in M(G \times \widehat{G}) $ for $p=1$ and $\mu \in L^p(\gtg)$ for $p>1$ such that $ T(f)=\mu \times f $ for all $f \in L^1(G \times \widehat{G})$.
			\end{enumerate}
			Also, for $1 \leq p \leq 2$, the above three are equivalent to the following:
			\begin{enumerate}
				
				\item[(iv)] there exists a unique measure $\mu \in M(G \times \widehat{G}) $  for $p=1$ and $\mu \in L^p(\gtg)$ for $p>1$ such that $ W(T(f))=W(\mu) W(f) $ for all $f \in L^1(G \times$ $\widehat{G})$.
				\item[(v)] there  exists a unique operator $M \in \mathcal{B}\left(L^2(G)\right) $ for $p=1$ and $M \in \mathcal{B}_{p'}\left(L^2(G)\right) $  such that $  W(T(f))=MW(f)  $ for all $f \in L^1(G \times \widehat{G})$.	
			\end{enumerate}
		\end{theorem}
		Using this characterization, any of the above statements can be used to define Weyl multipliers. Motivated by this, we define the concept of Weyl multiplier for vector-valued functions as follows.
		\begin{definition}
			Let $T: L^1(G \times \widehat{G};A) \to L^p(G \times \widehat{G};A)$ be a continuous linear operator, where $1 \leq p < \infty$. We say that $T$ is a Weyl multiplier if it satisfies the following conditions.
			\begin{enumerate}
				\item[(i)] $TT^t_{(x, \chi)}=T^t_{(x, \chi)}T $ for all $(x,\chi) \in G \times \widehat{G}$,
				\item[(ii)] $T(fa)=T(f)a$ for all $a \in A, f \in L^1(G \times \widehat{G};A) $.
			\end{enumerate}
			
		\end{definition}
			\begin{proposition}
			Let $f \in L^1(G \times \widehat{G};A)$ and $g \in L^{p}(G \times \widehat{G};A)$, $1 \leq p < \infty$,  then for all $(x,\chi) \in G \times \widehat{G}$, we have \label{lr1}
			\begin{equation} \label{lr}
				T^t_{(x, \chi)}(f \times g)= f \times T^t_{(x, \chi)}g. 
			\end{equation}
		\end{proposition}
		\begin{proof} 
			For $f \in L^1(G \times \widehat{G};A)$ and $g \in L^{p}(G \times \widehat{G};A)$, we have
			\begin{align*}
				T^t_{(x,\chi)}(f \times g)(x',\chi')=& f \times g (x'x^{-1},\chi'\overline{\chi}) \chi(x'x^{-1}) \\
				=& \int_{G \times \widehat{G}} f (x'x^{-1}x''^{-1},  \chi' \overline{\chi\chi''}) \chi''(x'x^{-1}x''^{-1}) g(x'',\chi'') \chi(x'x^{-1}) dx'' d \chi'' \\
				=& \int_{G \times \widehat{G}} f(x'x''^{-1}, \chi'\overline{\chi''}) \chi''\overline{\chi}(x'x''^{-1}) g(x''x^{-1},\chi''\overline{\chi})  \chi(x'x^{-1}) dx'' d \chi'' \\
				=& \int_{G \times \widehat{G}} f(x'x''^{-1}, \chi'\overline{\chi''}) \chi''(x'x''^{-1}) g(x''x^{-1},\chi''\overline{\chi})  \chi(x''x^{-1}) dx'' d \chi'' \\
				=& \int_{G \times \widehat{G}} f(x'x''^{-1}, \chi'\overline{\chi''}) \chi''(x'x''^{-1}) T^t_{(x,\chi)} g(x'', \chi'') dx'' d \chi'' \\ =& f \times T^t_{(x,\chi)}g (x', \chi'). \qedhere
			\end{align*}
		\end{proof}


		\section{Characterization of $L^1(G \times \widehat{G}; A)$ Weyl multipliers} In this section, we will study the vector-valued version of the $L^1$-multipliers.
		We will start by discussing an algebra structure in $M(\gtg; A)$. 	The following result follows as in the classical case; see \cite{White}.

		\begin{lemma} 
			The space
			$M(\gtg;A)$ becomes a Banach algebra under the twisted convolution defined as
			
				$$\mu \times \nu (f)=\int_{G  \times \widehat{G}} \int_{G  \times \widehat{G}} f \left(xx',\chi\chi'\right)  \chi'(x) d \mu(x,\chi)d \nu(x',\chi ').$$
			Also, the Banach algebra $L^1(\gtg;A)$ is an ideal in $M(\gtg;A)$. 
		\end{lemma}
		
		For every  $a \in A$, define an operator $M_a$ on $A$ as \begin{equation} \label{lmda} M_a(b):=ba, b \in A.\end{equation}
	
		Given two Banach spaces $X$ and $Y$, let $X\widehat{\otimes} Y$ denote the projective tensor product of $X$ and $Y$. If $X_i, \ Y_i$ $(i= 1,2)$ are Banach spaces and $T_i:X_i \to Y_i$ $ (i=1,2)$ are bounded linear operators then $T_1\widehat{\otimes} T_2(x_1 \otimes x_2) = T_1(x_1) \otimes T_2(x_2)$ defines a bounded linear operator from $X_1\widehat{\otimes}X_2$ to $Y_1\widehat{\otimes} Y_2 $.

		\begin{theorem}
			Let $A$ be a Banach algebra with a bounded approximate identity. Let $T: L^1(G \times \widehat{G};A) \to L^1(G \times \widehat{G};A)$ be a continuous linear operator. Then the following are equivalent:
			\begin{enumerate}[(i)]
				\item $T$ is a  Weyl multiplier.
				\item $	T  (M_a \wo T^t_{(x,\chi )}) =(M_a \wo T^t_{(x,\chi )})  T  \text{ for all }  a \in A \text{ and } (x,\chi) \in G \times \widehat{G}$ where $M_a$ is defined in equation \eqref{lmda}.
				\item $T(f \times g)= Tf \times g$ for all $f,g \in L^1(G \times \widehat{G};A)$.
				\item There exists a measure  $\nu$ in $M(G \times \widehat{G};A^{**})$ such that $Tf= \nu \times f$ for all $f \in L^1(G \times \widehat{G};A)$.	
			\end{enumerate}
			In addition, if $A$ is also reflexive, then the above conditions are equivalent to the following:
			\begin{enumerate}
				\item[(v)] There exists a measure $\nu \in M(G \times \widehat{G};A)$ such that
				$W(Tf)=W(\nu) W(f)$ for all $f \in L^1(G \times \widehat{G};A).$
				\item[(vi)] There exists $M \in \cb(L^2(G;A))$ such that $W(Tf) = M W(f) $ for all $f \in L^1(G \times \widehat{G};A).$
			\end{enumerate}
		\end{theorem}
		\begin{proof}
			$(i) \Rightarrow (ii)$.	Let $T$ be a Weyl multiplier. 
			For all $f \in L^1(G \times \widehat{G};A)$ and $(x, \chi) \in G \times \widehat{G}$, we have
			\begin{align*}
				T  (M_a \wo T^t_{(x,\chi )}) (f)=&  T(  T^t_{(x,\chi )}(f)a)\\=& T(  T^t_{(x,\chi )}(f))a\\=&     T^t_{(x,\chi )}(Tf)a\\=& (M_a \wo  T^t_{(x,\chi )})  T(f).
			\end{align*}
			
			$(ii) \Rightarrow (iii)$. For $a\in A$ and $x^* \in A^*$, define a map 
			$$\Phi_{a,x^*}(f)=x^*T(fa), \quad  f \in L^1(\gtg).$$
			Let $e_\beta$ be the approximate identity of $A$. Then

			\begin{align*}
				x^*((M_{e_\beta} \wo T^t_{(x,\chi)})T(fa)) =  & T^t_{(x,\chi)}(x^*( T(fa)e_\beta)) \\&\to  T^t_{(x,\chi)}(x^*( T(fa)))= T^t_{(x,\chi)}\Phi_{a,x^*} (f).
			\end{align*}
			Also,
			\begin{align*}
				x^*(T (M_{e_\beta} \wo T^t_{(x,\chi)})(fa))=x^*T(T^t_{(x,\chi)}( f a) e_\beta ) \to x^*T( T^t_{(x,\chi)}f a) =\Phi_{a,x^*}(T^t_{(x,\chi)} f).
			\end{align*}
			Thus the map $\Phi_{a,x^*}$ defined above commutes with twisted translations for all $a\in A$ and $x^* \in A^*$, hence by Theorem \ref{Weyl m}, for $f, g \in L^1(\gtg)$ and $a,b \in A$, we have
			\begin{align*}
				x^*(T(fa \times gb))=& x^*(T(f \times gab))=  x^*(T(fab))  \times g  
				\\=&x^*(T\left(M_b \wo T^t_{(e,\boldsymbol 1)}(fa)\right)) \times g=   x^*(T(fa)b) \times g=x^*(T(fa )\times gb) .\end{align*}
			Thus $T(fa \times gb) =T(fa) \times  gb$ and $(iii)$ follows.
			
			$(iii) \Rightarrow (iv)$.
			Let $\{g_\alpha\}$ be a bounded approximate identity of $L^1(G \times \widehat{G};A)$ with $\|g_\alpha\|=1$.  Then for $g \in L^1(G \times \widehat{G};A)$, we have
			\begin{equation}
				T(f ) = \lim_\alpha T(g_\alpha \times f).\label{cwm1}
			\end{equation}
			Also, $\{Tg_\alpha\}$ is a norm bounded subset of  $M(G \times \widehat{G};A^{**})\simeq C_0(\gtg,A^*)^*$. By Banach Alaoglu's theorem, there exists $\nu \in M(G \times \widehat{G};A^{**})$ such that $ \{Tg_\alpha\} \to \nu$ in weak$^*$ topology i.e.
			$$(Tg_\alpha(f))(x^*) \to \nu(f)(x^*)  \text{ for }f \in C_0(\gtg)\text{ and }x^* \in A^*.$$
Now for $a \in A, f,g \in C_\infty (\gtg)$ and each $x^* \in A^*$, we have 
   
			\begin{align*}
			&\left|x^* \int_{\gtg} f d\left(T\left(g_\alpha \times ag \right)\right)-\int_{\gtg} f d(\nu  \times ag)\left(x^*\right)\right|\\
			= & \left|x^* \left( \int_{\gtg} f d\left(T g_\alpha \times ag \right)\right )-\int_{\gtg} f d( \nu \times  ag)\left(x^*\right)\right| \\
			= & \bigg|\int_{\gtg}  \left[\int_{\gtg} f\left(xx',\chi\chi'\right)  \chi'(x)  dx^*\left( T g_\alpha a\right)(x,\chi)\right]  g(x',\chi')d x' d \chi '\\
			& -\int_{\gtg} \left[\int_{\gtg} f\left(xx',\chi\chi'\right)  \chi'(x) d(\nu \odot a)(x,\chi) dx d\chi\right]   \left(x^*\right) g(x',\chi') dx' d \chi'\bigg| \\
			= & \bigg| x^* \int_{\gtg} d (T g_\alpha a)(x,\chi) \int_{\gtg} f\left(xx',\chi\chi'\right)  \chi'(x) g(x',\chi') dx 'd \chi' \\
			& -\int_{\gtg} d(\nu \odot  a)(x,\chi) \int_{\gtg} f\left(xx',\chi\chi'\right)  \chi'(x) g(x',\chi') dx' d \chi' \left(x^*\right) \bigg| \\
			= & \left|x^* \int_{\gtg} h(x',\chi') d ( T g_\alpha a)(x',\chi')-(\nu \odot a)(h)\left(x^*\right)\right| \\
			=& \left| (a \odot x^* )(Tg_\alpha(h))-\nu(h)(a \odot x^*)\right| \to 0,		
		\end{align*}
		where $h(x',\chi')= f\left(xx',\chi\chi'\right)  \chi'(x) g(x',\chi')$ and $(\nu \odot a)(h)=\nu(h) \odot a$.

  Also, since $f \in C_\infty(\gtg)$, we have
			\begin{align*}  \left|  x^* \int_{\gtg} f d\left(T\left(g_\alpha \times ag \right)\right)-x^* \int_{\gtg} f d\left(T\left(a g\right)\right) \right| =&\left| x^*\big(T\left(g_\alpha \times ag \right)(f)-T\left(a g\right)(f)\big)\right| \\ \leq & \|x^*\|  \|f\|_\infty \|T(g_\alpha \times ag -ag)\|_1.
\end{align*}

Now, by using equation \eqref{cwm1}, we have	
			
			$$x^* \int_{\gtg} f d\left(T\left( g_\alpha \times a g \right)\right) \to x^* \int_{\gtg} f d\left(T\left(a g\right)\right).$$
			Therefore  for $a \in A, f,g \in C_\infty (\gtg)$ and each $x^* \in A^*$
			$$x^* \int_{\gtg} f d\left(T\left(a g\right)\right)=\int_{\gtg} f d(\nu \times ag )\left(x^*\right).$$

			Thus $Tf= \nu \times f $ for all $f \in L^1(\gtg;A)$.

			$(iv) \Rightarrow (i).$ For all $a \in A$ and $f \in L^1(\gtg;A)$ we have
			$$T(fa)= \nu \times fa =(\nu \times f)a= T(f)a.$$
			Similar to Proposition \ref{lr1},  it can be shown that
			$$T(T^t_{(x,\chi)})(f)= \nu \times  T^t_{(x,\chi)}f  =T^t_{(x,\chi)}(\nu  \times f)= T^t_{(x,\chi)}T(f)$$
			for all $(x,\chi) \in \gtg$. Thus $T$ is a Weyl multiplier.	
			
			$(iv) \Rightarrow (v)$. If $A$ is reflexive, there exists a measure $\nu \in M(G \times \widehat{G};A)$ such that $Tf= \nu \times f$ for all $f \in L^1(G \times \widehat{G};A)$. Then $W(Tf)=W(\nu \times f)= W(\nu) W(f)$.\\
			$(v) \Rightarrow (vi)$ is obvious by taking $M=W(\nu)$. \\
			$(vi) \Rightarrow (iii) $. It can be seen that for all $f,g \in L^1(G \times \widehat{G}; A)$, we have $$ W(T(f \times g))=MW(f \times g)=MW(f) W (g)=W(Tf)W(g)= W(Tf \times g).$$  Therefore $T(f \times g) = Tf \times g$ as required.
		\end{proof}
		
			\section{Characterization of $L^1(G \times \widehat{G}; A)$-$L^p(G \times \widehat{G}; A)$ Weyl multipliers}  In this section, we will study the vector-valued version of the $L^1$-$L^p$ Weyl multipliers under the assumption that $A$ is commutative Banach algebra with bounded approximate identity.
			
			For $1 \leq p < \infty,$ any function $g \in L^{p'}(\gtg; A^*)$ defines a linear functional $\phi_g \in\left(L^p(\gtg ; A)\right)^*$ by the formula
		$$
		\left\langle f, \phi_g\right\rangle:=\int_{\gtg}\langle f(x,\chi), g(x,\chi)\rangle dx d \chi, f \in L^p(\gtg ; A) .
		$$
		Recall that a set $Y \subseteq  A^*$ is said to be norming for $A$ if $\underset{f \in Y  \setminus \{0\}}{\sup} \frac{|f(x)|}{\|f\|_{A^*}}=\|x\|_A.$
		
		\begin{proposition} \label{dual}
			Let $1 \leq p < \infty$. The mapping $g \mapsto \phi_g$ defines an isometry of $L^{p'}(\gtg ; A^*)$ onto a closed subspace of $\left(L^p(\gtg ; A)\right)^*$ which is norming for $L^p(\gtg ; A)$.
		\end{proposition}
		We will denote $\phi_g$ by $g$ when the notation is obvious.

		To study the characterization of the vector-valued case of Weyl multipliers, we first need to define a product $\divideontimes$ for $f \in L^1(G \times \widehat{G}; A)$ and $ \nu \in  \left(L^{p}(G \times \widehat{G}; A^*)\right)^*,1 <p< \infty $ such that if $\nu \in L^{p'}(G \times \widehat{G}; A)$, then $\nu  \divideontimes f = \nu \times f$. For this, we will use duality as follows.
		
		For $f \in L^1(G \times \widehat{G};A), g \in L^p(G \times \widehat{G};A)$ and $h \in L^{p'}(G \times \widehat{G};A^*)$, consider
		\begin{align*} \left \langle  f \times g, h \right \rangle  =&  \int_{G \times \widehat{G}} \left \langle  f \times g (x, \chi), h(x,\chi) \right \rangle   dx d \chi \\
		=&  \int_{G \times \widehat{G}} \left \langle  \int_{G \times \widehat{G}} f(x',\chi')L^t_{(x',\chi')}g(x,\chi) dx' d\chi', h(x,\chi)\right \rangle   dx d\chi \\
		=& \int_{G \times \widehat{G}} \left( \int_{G \times \widehat{G}} \left \langle  f(x',\chi')L^t_{(x',\chi')}g(x,\chi),h(x,\chi)\right \rangle  dx d \chi  \right) dx' d \chi' \\
		=& \int_{G \times \widehat{G}} \left(  \left \langle  f(x',\chi'), \int_{G \times \widehat{G}} L^t_{(x',\chi')}g(x,\chi) \odot h(x,\chi) dx d \chi  \right \rangle   \right) dx' d \chi' .
		\end{align*}
		So, if we define $g \oplus h(x',\chi'):=\int_{G \times \widehat{G}} L^t_{(x',\chi')}g(x,\chi) \odot h(x,\chi) dx d \chi$, we get
		$$\left \langle  f \times g, h \right \rangle  = \left \langle  f, g \oplus h\right \rangle .	 $$
		\begin{lemma}
		For $f \in L^1(G \times \widehat{G};A)$ and $g \in L^p(G \times \widehat{G};A^*), 1 < p < \infty$, define
		$$f \oplus g(x',\chi'):=\int_{G \times \widehat{G}} L^t_{(x',\chi')}f(x,\chi) \odot g(x,\chi) dx d \chi.$$
		Then $f \oplus g \in L^p(G \times \widehat{G};A^*)$.
		\end{lemma}
		\begin{proof}
		For $f \in L^1(G \times \widehat{G};A)$ and $g \in L^p(G \times \widehat{G};A^*)$, we have
		\begin{align*}
			\|	f \oplus g\|^p_{L^p(G \times \widehat{G};A^*)}  =&  \int_{G \times \widehat{G}}\|f \oplus g(x,\chi)\|^p_{A^*} dx d\chi  \\
			\leq & \int_{G \times \widehat{G}}\left ( \int_{G \times \widehat{G}}  \|f(xx'^{-1},\chi \overline{\chi'})\|_A^{p-1} (\|f(xx'^{-1},\chi \overline{\chi'})\|_A \|g(x,\chi)\|_{A^*}^p) dx' d \chi' \right ) dx d\chi 
			\\ \leq & \|f\|_{L^1(G \times \widehat{G};A)}^{p/ p'} \int_{G \times \widehat{G}}\left ( \int_{G \times \widehat{G}}   (\|f(xx'^{-1},\chi \overline{\chi'})\|_A \|g(x,\chi)\|_{A^*}^p) dx' d \chi' \right ) dx d\chi.
		\end{align*}
		Now by applying Fubini's theorem, we get
		$$	\|	f \oplus g\|_{L^p(G \times \widehat{G};A^*)} \leq \|f\|_{L^1(G \times \widehat{G};A)} \|g\|_{L^p(G \times \widehat{G};A^*)}.  $$ 
		\end{proof}
		\begin{definition}
		For $1 <p< \infty $, let $f \in L^1(G \times \widehat{G};A)$ and $ \nu \in  \left(L^{p'}(G \times \widehat{G};A^*)\right)^*$. Define 
		$$\left \langle  \nu \divideontimes f, h \right \rangle  := \nu( f \oplus h)  \  \forall  \   h  \in L^{p'}(G \times \widehat{G};A^*).$$
		\end{definition}
		\begin{remark} \label{reason}
		The way $f \oplus h$ is constructed,  it is clear that if in particular $\nu \in L^p(G \times \widehat{G}; A)$, then $\nu  \divideontimes f = \nu \times f$.
		\end{remark}
		
		For every  $a \in A$, define an operator $\Lambda_a$ on $A$ as \begin{equation} \label{lmda} \Lambda_a(f):=af, f \in L^p(\gtg;A).\end{equation}
		\begin{proposition} \label{prop}
		Let $A$ be commutative Banach algebra with an approximate identity. For $1 <p< \infty, $  let $T:L^1(G \times \widehat{G};A) \to L^p(G \times \widehat{G};A)$ be a continuous linear operator.  Then $T$ is a Weyl multiplier iff $T$ satisfies the following:
		\begin{equation} \label{eqc}
			T  (\Lambda_a  T^t_{(x,\chi )}) =(\Lambda_a T^t_{(x,\chi )}) T  \  \forall  a \in A \text{ and } (x,\chi) \in G \times \widehat{G}.
		\end{equation}
		\end{proposition} 
		\begin{proof}
		Let $T$ be a Weyl multiplier. Then for all $f \in L^1(G \times \widehat{G};A)$ and $(x, \chi) \in G \times \widehat{G}$, we have
		\begin{align*}
			T  (\Lambda_a T^t_{(x,\chi )}) (f)=  T(a  T^t_{(x,\chi )}(f))= aT(  T^t_{(x,\chi )}(f))=   a  T^t_{(x,\chi )}(Tf)=(\Lambda_a  T^t_{(x,\chi )})  T(f).
		\end{align*}
		
		Conversly, for $a \in A$, we have 
		\begin{equation} \label{eqx}
			T(af)=T((\Lambda_a  T^t_{(e,\mathbf{1})})(f))= (\Lambda_a T^t_{(e,\mathbf{1})})(T(f)) = a T(f).
		\end{equation}

		Now, let $\{f_\alpha\}_{\alpha \in \Delta }$ and $ \{e_\beta\}_{\beta \in M }$ be approximate identities of $L^1(G \times \widehat{G})$ and $A$ respectively so that $\{e_\beta f_\alpha \}_{(\alpha,\beta) \in  \Delta \times  M }$ is an approximate identity for $L^1(G \times \widehat{G};A)$. Then, by using equation \eqref{lr}, for all $f \in L^1(G \times \widehat{G};A)$, we have 
		$$T\left( e_\beta T^t_{(x, \chi)}(f_\alpha \times f) \right) =  T(e_\beta f_\alpha \times T^t_{(x,\chi)}f) .$$
		Hence,
		\begin{align*}
			\|T\left( e_\beta T^t_{(x, \chi)}(f_\alpha \times f) \right) - T(T^t_{(x, \chi)}f)\|_{L^p(G \times \widehat{G};A)}=& \|T(e_\beta f_\alpha \times T^t_{(x,\chi)}f-T^t_{x, \chi)}f) \|_{L^p(G \times \widehat{G};A)} \\ 
			\leq & \|T\| \|e_\beta f_\alpha \times T^t_{(x,\chi)}f - T^t_{(x,\chi)}f \|_{L^1(G \times \widehat{G};A)}.
		\end{align*}
		Since $\{e_\beta f_\alpha \}_{(\alpha,\beta) \in  \Delta \times  M }$  is an approximate identity for $L^1(G \times \widehat{G};A)$, we have
		\begin{equation} \label{eq1}
			T(T^t_{(x, \chi)}f) = \lim_{\alpha, \beta} 	T\left( e_\beta T^t_{(x, \chi)}(f_\alpha \times f) \right)   \    \forall f \in L^1(G \times \widehat{G};A) .
		\end{equation} 
		Similarly for all $f \in L^1(G \times \widehat{G};A)$, we have 
		\begin{align*}
			\| e_\beta T^t_{(x, \chi)}( T ( f_\alpha \times f))  - T^t_{(x, \chi)}(Tf)\|_{L^p(G \times \widehat{G};A)}=& \|  T^t_{(x, \chi)}(  e_\beta T ( f_\alpha \times f))  - T^t_{(x, \chi)}(Tf)\|_{L^p(G \times \widehat{G};A)}\\ =& 	\|  T^t_{(x, \chi)}\left(  e_\beta T ( f_\alpha \times f)-Tf \right)\|_{L^p(G \times \widehat{G};A)}\\ =& 	\|   T (  e_\beta f_\alpha \times f)-Tf \|_{L^p(G \times \widehat{G};A)}\\ \leq & \|T\|\| e_\beta f_\alpha \times f-f\|_{L^1(G \times \widehat{G};A)}.
		\end{align*}
		Thus,
		\begin{equation} \label{eq2}
			T^t_{(x, \chi)}(Tf)	 =\lim_{\alpha, \beta} e_\beta T^t_{(x, \chi)}( T ( f_\alpha \times f))  \    \forall f \in L^1(G \times \widehat{G};A).
		\end{equation} 
		Now using equations \eqref{eqc}, \eqref{eq1} and \eqref{eq2}, we get $TT^t_{(x,\chi)}= T^t_{(x,\chi)}T $ for all $f \in L^1(G \times \widehat{G};A)$.
		\end{proof}
		The following theorem characterizes the vector-valued Weyl multipliers in terms of twisted convolution.
		\begin{theorem} \label{main}
		Let $A$ be a commutative Banach algebra with a bounded approximate identity. For $1 < p < \infty$, let  $T: L^1(G \times \widehat{G};A) \to L^p(G \times \widehat{G};A)$ be a continuous linear operator. Then the following are equivalent:
		\begin{enumerate}
			\item[(i)] $T$ is a Weyl multiplier.
			\item[(ii)] $	T  (\Lambda_a  T^t_{(x,\chi )}) =(\Lambda_a  T^t_{(x,\chi )})  T  \text{ for all }  a \in A \text{ and } (x,\chi) \in G \times \widehat{G}$ where $\Lambda_a$ is defined in equation \eqref{lmda}.
			\item[(iii)] $T(f \times g)= Tf \times g$ for all $f,g \in L^1(G \times \widehat{G};A)$.
			\item [(iv)] There exists an element   $\nu$ in $\left(L^{p'}(G \times \widehat{G};A^{*})\right)^*$ such that $Tf= \nu \divideontimes f$ for all $f \in L^1(G \times \widehat{G};A)$.		
		\end{enumerate}
		\end{theorem}
		\begin{proof}
		$(i) \Leftrightarrow (ii)$ follows from Proposition \ref{prop}.
		
		
		$(i) \Rightarrow (iii)$.  For $ f \in L^1(G \times \widehat{G};A)$ and $ \nu \in M(G \times \widehat{G})$, rewriting the proof of $(a) \Rightarrow (b)$ of \cite[Theorem 5.2]{RSK} , we can show that 
		$$\int_{G \times \widehat{G}} \left \langle  T(f \times \nu)(x,\chi),h(x,\chi)\right \rangle  dx d\chi= \int_{G \times \widehat{G}}  \left \langle  (Tf \times \nu)(x,\chi),h(x,\chi)\right \rangle  dx d\chi,$$
		for all $h \in L^{p'}(G \times \widehat{G};A^*)$.
		Hence using Proposition \ref{dual}, we get $ T(f \times \nu)= Tf \times \nu$ for all $f \in L^1(G \times \widehat{G};A)$ and $ \nu \in M(G \times \widehat{G}).$
		
		Now, for $a \in A, \varphi \in L^1(G \times \widehat{G}),$ let $g=a\varphi$. If we define $d \nu(x,\chi):= \varphi(x,\chi) dx d\chi$, then $\nu \in M(G \times \widehat{G})$.
		Hence $$Tf \times g = Tf \times a \varphi=a (Tf \times \nu) =aT(f \times \nu)= T(f \times g).$$
		Now using the density of elementary tensors in $L^1(G \times \widehat{G}) \widehat{\otimes} A$ and the fact that $L^1(G \times \widehat{G})\widehat{\otimes} A \cong L^1(G \times \widehat{G};A)$, we have $Tf \times g =T(f \times g)$ for all $ f , g \in L^1(G \times \widehat{G};A)$.
		
		$(iii) \Rightarrow (iv)$. 
		Let $\{g_\alpha\}$ be a bounded approximate identity of $L^1(G \times \widehat{G};A)$ with $\|g_\alpha\|=1$. Then for $f \in L^1(G \times \widehat{G};A), Tf = \lim_{\alpha} Tg_\alpha \times f $ in $L^p(G \times \widehat{G};A)$ norm. Therefore
		\begin{equation} \label{1} \left \langle  Tf,h \right \rangle  = \lim_{\alpha} \left \langle  Tg_\alpha \times f, h \right \rangle   \  h \in L^{p'}(G \times \widehat{G};A^*).
		\end{equation} Also since 
		$$\|Tg_\alpha\|_{L^p(G \times \widehat{G};A)} \leq \|T\|\|g_\alpha\|_{L^1(G \times \widehat{G};A)},$$
		$\{Tg_\alpha\}$ is a norm bounded subset of  $L^p(G \times \widehat{G};A)$ and hence of  $L^p(G \times \widehat{G};A^{**})$. Since $L^p(G \times \widehat{G};A^{**}) \hookrightarrow \left(L^{p'}(G \times \widehat{G};A^{*}) \right) ^* $, by Banach Alaoglu's theorem, there exists a  subnet $\{Tg_\beta\}$ of $\{Tg_\alpha\}$ and a $\nu \in \left(L^{p'}(G \times \widehat{G};A^{*}) \right)^*$ such that $ \{Tg_\beta\} \to \nu$ in weak$^*$ topology. 
		Then for all $h \in L^{p'}(G \times \widehat{G};A^*)$, we have
		\begin{align}
			\nonumber	\left \langle  \nu \divideontimes f,h \right \rangle  = \nu (f \oplus h) =& \lim_\beta Tg_\beta(f \oplus h) \\=&  \label{2} \lim_\beta \left \langle  Tg_\beta \divideontimes f, h \right \rangle . 		
		\end{align}
		
		As $Tg_\beta \in L^p(G \times \widehat{G};A)$, using equations \eqref{1}, \eqref{2} and Remark \ref{reason}, we have
		$$ \left \langle  Tf,h \right \rangle  = \left \langle  \nu \divideontimes f,h \right \rangle   ,  h \in L^{p'}(G \times \widehat{G};A^*).$$
		Hence $Tf= \nu \divideontimes f$ for all $f \in L^1(G \times \widehat{G};A)$ as required.
		
		$(iv) \Rightarrow (i).$ For $(x, \chi)  \in G \times \widehat{G}$ and $f \in L^1(G \times \widehat{G};A)$ and $h \in L^{p'}(G \times \widehat{G};A^*)$  we first claim that 
		$$T^t_{(x,\chi)}f \oplus h = f \oplus \tilde{h}_{(x,\chi)},\text{ where } \tilde{h}_{(x,\chi)}(x',\chi')=h(x'x,\chi'\chi)\chi(x').$$
		For $g \in L^p(G \times \widehat{G};A)$, we have
		\begin{align*}&\int_{G \times \widehat{G}} \left \langle T^t_{(x,\chi)}f \oplus h(x',\chi') , g(x',\chi') \right \rangle dx' d\chi'\\=&
			\int_{G \times \widehat{G}} \left \langle \int_{G \times \widehat{G}} L^t_{(x',\chi')}T^t_{(x,\chi)}f(x'',\chi'') \odot h(x'',\chi'') dx'' d \chi'', g(x',\chi') \right \rangle dx'd\chi'
			\\=& \int_{G \times \widehat{G}} \left \langle \int_{G \times \widehat{G}} L^t_{(x',\chi')}T^t_{(x,\chi)}f(x'',\chi'') g(x',\chi') dx'd\chi' , h(x'',\chi'') \right \rangle dx'' d \chi'' 
			\\=& \int_{G \times \widehat{G}} \left \langle \int_{G \times \widehat{G}} T^t_{(x,\chi)}f(x''x'^{-1},\chi''\overline{\chi'}) \chi''\overline{\chi'}(x') g(x',\chi') dx'd\chi' , h(x'',\chi'') \right \rangle dx'' d \chi'' 
			\\=& \int_{G \times \widehat{G}} \left \langle \int_{G \times \widehat{G}} f(x''x'^{-1}x^{-1},\chi''\overline{\chi'\chi}) \chi(x''x'^{-1}x^{-1}) \chi''\overline{\chi'}(x') g(x',\chi') dx'd\chi' , h(x'',\chi'') \right \rangle dx'' d \chi'' .
		\end{align*}
		Now taking the map $(x'',\chi'') \to (x''x,\chi''\chi)$, we get
		\begin{align*}
			&\int_{G \times \widehat{G}} \left \langle T^t_{(x,\chi)}f \oplus h(x',\chi') , g(x',\chi') \right \rangle dx' d\chi'\\
			=& \int_{G \times \widehat{G}} \left \langle \int_{G \times \widehat{G}} f(x''x'^{-1},\chi''\overline{\chi'}) \chi(x''x'^{-1}) \chi''\chi\overline{\chi'}(x') g(x',\chi') dx'd\chi' , h(x''x,\chi''\chi) \right \rangle dx'' d \chi''. \\
			=& \int_{G \times \widehat{G}} \left \langle \int_{G \times \widehat{G}} \left(f(x''x'^{-1},\chi''\overline{\chi'}) \chi''\overline{\chi'}(x')\right) \chi(x'')  g(x',\chi') dx'd\chi' , h(x''x,\chi''\chi) \right \rangle dx'' d \chi'' \\
			=& \int_{G \times \widehat{G}} \left \langle \int_{G \times \widehat{G}} L^t_{(x',\chi')}f(x'',\chi'')\chi(x'')  g(x',\chi') dx'd\chi' , h(x''x,\chi''\chi) \right \rangle dx'' d \chi'' \\
			=& \int_{G \times \widehat{G}} \left \langle \int_{G \times \widehat{G}} L^t_{(x',\chi')}f(x'',\chi'') \odot \tilde{h}_{(x,\chi)}(x'',\chi'') dx'' d \chi'', g(x',\chi') \right \rangle dx'd\chi'\\
			=& \int_{G \times \widehat{G}} \left \langle f \oplus \tilde{h}_{(x,\chi)}(x',\chi') , g(x',\chi') \right \rangle dx' d\chi'	,	
		\end{align*}
		where $\tilde{h}_{(x,\chi)}(x'',\chi'')=h(x''x,\chi''\chi)\chi(x'').$ Hence the claim.
		
		Now for $f \in L^1(G \times \widehat{G};A)$, using the claim, it can be shown that 
		\begin{align*}
			\int_{G \times \widehat{G}} \left \langle  TT^t_{(x,\chi)}f(x',\chi'), h(x',\chi') \right \rangle  dx' \chi' =& \int_{G \times \widehat{G}} \left \langle  \nu \divideontimes T^t_{(x,\chi)}f(x',\chi'), h (x',\chi') \right \rangle   dx' d\chi' \\
			=& \int_{G \times \widehat{G}} \left \langle  \nu \divideontimes f(x',\chi'), \tilde{h}_{(x,\chi)}(x',\chi') \right \rangle   dx' d\chi' \\
			=& \int_{G \times \widehat{G}} \left \langle  Tf(x',\chi'),h(x'x,\chi'\chi)\chi(x') \right \rangle   dx' d\chi' \\	
			=& \int_{G \times \widehat{G}} \left \langle  Tf(x'x^{-1},\chi'\overline{\chi}),h(x',\chi')\chi(x'x^{-1}) \right \rangle   dx' d\chi' \\	
			=& \int_{G \times \widehat{G}} \left \langle  T^t_{(x,\chi)} Tf(x',\chi'),h (x',\chi') \right \rangle   dx' d\chi',				
		\end{align*}
		for all $h \in L^{p'}(G \times \widehat{G};A^*)$. Therefore $TT^t_{(x,\chi)}=T^t_{(x,\chi)}T$  for every $(x,\chi) \in G \times \widehat{G}$.
		
		Now we have to show that for $a \in A$, $T(af)=aT(f).$
		By using the commutativity of $A$, it can be easily seen that for all $h \in L^{p'}(G \times \widehat{G};A^*)$, we have
		\begin{align*}
			\int_{G \times \widehat{G}} \left \langle  aT(f)(x, \chi), h(x,\chi) \right \rangle  dx d\chi =&  \int_{G \times \widehat{G}} \left \langle  a(\nu \divideontimes f)(x, \chi), h(x,\chi) \right \rangle  dx d\chi \\
			=& \int_{G \times \widehat{G}} \left \langle  \nu \divideontimes f(x, \chi), a \odot h(x,\chi) \right \rangle  dx d\chi. 
		\end{align*}
		Again, by using commutativity, we get  $\nu(af \oplus h)= \nu (f \oplus h_a)$ where $h_a(x,\chi)= a\odot h(x,\chi)$. Therefore
		\begin{align*}
			\int_{G \times \widehat{G}} \left \langle  T(af)(x, \chi), h(x,\chi) \right \rangle  dx d\chi  =& \int_{G \times \widehat{G}} \left \langle  \nu \divideontimes af(x, \chi), h(x,\chi) \right \rangle  dx d\chi \\
			=& \nu (af \oplus h) \\ =& \nu(f \oplus h_a)\\=& \int_{G \times \widehat{G}} \left \langle  aT(f)(x, \chi), h(x,\chi) \right \rangle  dx d\chi 
		\end{align*}
		for all	$h \in L^{p'}(G \times \widehat{G};A^*)$.  Using duality, $T(af)=aT(f)$ for every $a \in A$.
		\end{proof}
	\section*{Acknowledgement}
	The authors thank Prof. R. Radha (IIT Madras) for giving some valuable suggestions. The first-named author
	(RS) wishes to thank the Graduate Aptitude Test in Engineering, India for its research fellowship.
	\section*{Competing Interests}
	The authors declare that they have no competing interests.
	\bibliographystyle{acm}
	\bibliography{referCWM.bib}
\end{document}